\documentclass[10pt]{article}
\usepackage{color,amsmath,amsthm,amssymb}
\usepackage{palatino,wrapfig,verbatim,colortbl,hyperref,sectionbox}
\usepackage{graphicx}
\usepackage[english]{babel}
\usepackage[utf8x]{inputenc}
\usepackage{epsfig}
\usepackage{setspace}
\usepackage{multirow,bigdelim}
\usepackage{arydshln}
\newtheorem{thm}{Theorem}

\title{ Covering Array Bounds Using Analytical Techniques}
\author{Ruyue (Julia) Yuan\\Department of Mathematics and Computer Science\\Valparaiso University\\Valparaiso, IN 46383 \and Zoe Koch\\Department of Mathematics\\University of Utah\\Salt Lake City, UT 84112\and Anant Godbole\\Department of Mathematics and Statistics\\East Tennessee State University\\Johnson City, TN 37614}
\date{}

\begin{document}
\def\lg{{\rm lg}}
\def\lr{\left(}
\def\rr{\right)}
\maketitle

\begin{abstract}

A $t$-covering array with entries from the alphabet ${\cal Q}=\{0,1,\ldots,q-1\}$ is a $k\times n$ stack, so that for any choice of $t$ (typically non-consecutive) columns, each of the $q^{t}$ possible $t$-letter words over ${\cal Q}$ appear at least once  among the rows of  the selected columns. We will show how a combination of the Lov\'asz local lemma; combinatorial
analysis; Stirling's formula; and Calculus enables one to find better asymptotic bounds for
the minimum size of $t$-covering arrays, notably for $t = 3, 4$.  Here size is measured in the number of rows, as expressed in terms of the number of columns. 
\end{abstract}

\section{Introduction}

A $t$-covering array with entries from the alphabet ${\cal Q}=\{0,1,\ldots,q-1\}$ is a $k\times n$ stack, so that for any choice of $t$ typically non-consecutive columns, each of the $q^{t}$ possible $t$-letter words over ${\cal Q}$ appear at least once  among the rows of  the selected columns. The following problem is central; see, e.g., \cite{colbourn}, \cite {sloane}:  Given the parameters $q, t$, what is the smallest $k$ for which a covering array with these parameters exists?  Specifically, we seek a function $k_0=k_0(n)=k_0(n,q,t)$ such that as $n\to\infty$, $k\ge k_0(n)\Rightarrow$ a $t$-covering array exists.  Sperner's theorem was used by Kleitman and Spencer (see \cite{sloane}) to give a very satisfactory answer for $t=q=2$, while the work of Roux (again, see \cite{sloane}) showed that for $t=3; q=2$, we have
\begin{equation}k_0(n,2,3)=7.65\lg n(1+o(1)),\end{equation}
where, here and throughout this paper, $\lg:=\log_2$.  A general upper bound of 
\begin{equation}k_0(n,q,t)=(t-1)\frac{\lg n}{\lg\left(q^t/(q^t-1)\right)}(1+o(1))\end{equation}
was produced in \cite{gss}.  Notice that plugging in $q=2, t=3$ in (2) yields a bound of
\[k_0(n,2,3)=10.3\lg n(1+o(1)),\]
which shows that the general bounds of \cite{gss} are inferior to the specific bound in (1), which was obtained by employing random methods with equal weight columns (an equal number of zeros and ones in each column in the binary case) either without (Roux) or with (\cite{gss}) the use of the Lov\'asz local lemma.  Some improvement in (2) was made in the \cite{dg}, where a ``tiling method" was employed.  In this paper, we adapt the methods of Roux (\cite{sloane}) and \cite{gss} to improve the bounds in (2) for several other cases.  The analysis is difficult but not daunting for the cases we consider: a combination of the Lov\'asz local lemma (see, e.g., \cite{alon}); elementary combinatorial
analysis; Stirling's formula; and Calculus is employed to obtain our new results.  
The case of $t=3, q\ge 3$ is considered in Section 2. We turn our attention to $t=4, q=2$, where double sums need to be employed, in Section 3. 

\section{The Case of $t=3$}
\label{sec:examples}

\subsection{$q=3:  3$-Covering Arrays with a Three-Letter Alphabet}

\begin{thm}
$$k_0(n,3,3)\le 32.03 \cdot \lg(n)(1+o(1)).$$
\end{thm}

\begin{proof}
Let $n=3m$, and let us randomly place $m$ of each of the letters 0, 1, and 2 in each of the $k$ columns. The probability that any one set of three columns is missing any one  of the 27 ternary three letter words, say 111, is
$$p={{{3m}\choose{m}}\frac{\sum_{j=0}^{m} {m \choose j} \cdot {2m \choose m-j} \cdot {3m-j \choose m}}{{3m \choose m}^{3}}}={\frac{\sum_{j=0}^{m} {m \choose j} \cdot {2m \choose m-j} \cdot {3m-j \choose m}}{{3m \choose m}^{2}}}.$$
This expression is derived as follows:  First place the $m$ ones in the first column in ${{3m}\choose{m}}$ ways.  Then, for some $j$, we pick $j$ of the spots in these $m$ positions to have a $1$ in the second column.  Finally, since the word 111 is to be absent, the $m$ ones in column 3 all have to be in the $3m-j$ spots where the first two columns' entries are not both 1.  The union bound now tells us that the probability $\pi$ that at least one word is missing in any set of three columns is given by 
\[\pi\le27p.\]
Next, we maximize the numerator summand in the expression for $p$ by parametrizing:  Set $j=Am$ for some $0\le A\le1$ and use Stirling's approximation to get (with $C$ representing a generic constant):

\begin{align}
&{m \choose j} \cdot {2m \choose m-j} \cdot {3m-j \choose m}\nonumber\\
=&\frac{m!}{j! (m-j)!} \cdot \frac{(2m)!}{(m-j)! (m+j)!} \cdot \frac{(3m-j)!}{m! (2m-j)!}\nonumber\\
=&\frac{(2m)! (3m-j)!}{j! (m-j)! (m-j)! (m+j)! (2m-j)!}\nonumber\\
\le&\frac{C}{m^{3/2}}\left(\frac{2m}{e}\right)^{2m} \cdot \left(\frac{(3-A)m}{e}\right)^{(3-A)m} \cdot \left(\frac{e}{Am}\right)^{Am} \cdot \left(\frac{e}{(1-A)m}\right)^{2(1-A)m} \nonumber\\
&\cdot \left(\frac{e}{(1+A)m}\right)^{(1+A)m}\cdot \left(\frac{e}{(2-A)m}\right)^{(2-A)m}\nonumber\\
=&\frac{C}{m^{3/2}}\left[\frac{2^{2} \cdot (3-A)^{(3-A)}}{A^{A} \cdot (1-A)^{2(1-A)} \cdot (1+A)^{(1+A)} \cdot (2-A)^{(2-A)}}\right]^{m}.
\end{align}
In order to find the critical value of $A$ in the exponential part of (3), we will maximize
$q(A)=\ln 4+(3-A)\ln(3-A)-A\ln(A)-2(1-A)\ln(1-A)-(1+A)\ln(1+A)-(2-A)\ln(2-A).$
We have:
\begin{align*}
q'(A)=&-\frac{3-A}{3-A}-\ln(3-A)-\left[\frac{A}{A}+\ln(A)\right]-\left[-\frac{2(1-A)}{1-A}-2\ln(1-A)\right]\\
&-\left[\frac{1+A}{1+A}+\ln(1+A)\right]-\left[-\frac{2-A}{2-A}-\ln(2-A)\right]\\
=&-\ln(3-A)-\ln(A)+2 \cdot \ln(1-A)-\ln(1+A)+\ln(2-A)\\
=&\ln\left(\frac{(1-A)^{2} \cdot (2-A)}{(3-A) \cdot A \cdot (1+A)} \right).
\end{align*}
Setting $q'(A)=0$, we see that $A=2-\sqrt{3}$.
Plugging $A=2-\sqrt{3}$ into (3), we see that for each $j$,
\begin{eqnarray}&&{m \choose j} \cdot {2m \choose m-j} \cdot {3m-j \choose m}\nonumber\\
&\le&
\frac{C}{m^{3/2}}\left[\frac{2^{2} \cdot (1+\sqrt{3})^{(1+\sqrt{3})}}{(2-\sqrt{3})^{(2-\sqrt{3})} \cdot (\sqrt{3}-1)^{(\sqrt{3}-1)} \cdot (3-\sqrt{3})^{(3-\sqrt{3})} \cdot \sqrt{3}^{\sqrt{3}}}\right]^{m}\nonumber\\
&\approx & \frac{C}{m^{3/2}}40.0148^{m}.
\end{eqnarray}
Next, we use Stirling's Approximation to estimate the denominator in the expression for $p$:
$$\frac{(3m)!}{(2m)!(m)!}\ge\frac{C}{m^{1/2}}\frac{\left(\frac{3m}{e}\right)^{3m}}{\left(\frac{2m}{e}\right)^{2m} \cdot \left(\frac{m}{e}\right)^{m}}=\frac{C}{m^{1/2}}\left(\frac{27}{4}\right)^{m}.$$
Thus, on bounding the numerator of the expression for $p$ by $m$ times the maximum summand, we get $$\pi\le C\sqrt{m}\frac{40.0148^{m}}{\left(\frac{27}{4}\right)^{2m}}.$$
Now whether or not a given set of three columns is missing at least one word depends on $O(n^2)$ other sets of columns, namely the ones that share at least one column with the given set.  Thus the dependence number $d$ in the Lov\'asz lemma is of magnitude $n^2$.  The lemma states that if $e\pi d<1$ then the probability that we have no sets of such deficient columns is positive, i.e. a construction exists that satisfies the criteria of a covering array.  Now the inequality $e\pi d<1$ may be seen to hold, using elementary algebra, if 
$$m>\frac{2 \lg(n)}{\lg(1.138)}(1+o(1)) \approx 10.67 \lg(n)(1+o(1)),$$
or
$$k=3m>32.03 \lg(n)(1+o(1)).$$  It follows that $k_0\le 32.03 \lg(n)(1+o(1))$, as claimed.
\end{proof}
\noindent REMARKS:  The general bound in (2) yields $k_0(n,3,3)\le 36.73\lg n$, so we have quite an improvement.  Notice also that the exact values of the constants $C$ and the exact nature of the polynomial terms in Stirling's approximation did not affect the end asymptotic result (even though a more careful analysis {\it would} be needed for bounds for specific values of $k$.)  Accordingly, in the rest of the paper we will not explicitly mention these terms, and use Stirling's approximation as 
\[N!\cong\lr\frac{N}{e}\rr^N,\]
where $f(n)\cong g(n)$ will mean that $f(n)$ is bounded both above and below by some rational quantity times $g(n)$.

\subsection{$q=4:  3$-covering Arrays with a Four-letter Alphabet}
\begin{thm}
$$k_0(n,4,3)\le81.28 \cdot \lg(n)(1+o(1)).$$
\end{thm}

\begin{proof}  The proof is very similar to that of Theorem 1.  
We first find the expression of the probability $p$ of avoiding a particular word in an array of size $4m\times n$, where each column contains an equal number of randomly placed letters 0, 1, 2, and 3.  We have
$$p={\frac{\sum_{j=0}^{m} {m \choose j} \cdot {3m \choose m-j} \cdot {4m-j \choose m}}{{4m \choose m}^{2}}}.$$
We then maximize the summand in the numerator:
\begin{align*}
&{m \choose j} \cdot {3m \choose m-j} \cdot {4m-j \choose m}\\
=&\frac{m!}{j! (m-j)!} \cdot \frac{(3m)!}{(m-j)! (2m+j)!} \cdot \frac{(4m-j)!}{m! (3m-j)!}\\
=&\frac{(3m)! (4m-j)!}{j! (m-j)! (m-j)! (2m+j)! (3m-j)!}\\
\cong&\left(\frac{3m}{e}\right)^{3m} \cdot \left(\frac{(4-A)m}{e}\right)^{(4-A)m} \cdot \left(\frac{e}{Am}\right)^{Am} \cdot \left(\frac{e}{(1-A)m}\right)^{2(1-A)m}\\ &\cdot \left(\frac{e}{(2+A)m}\right)^{(2+A)m}
\cdot \left(\frac{e}{(3-A)m}\right)^{(3-A)m}\\
=&\left[\frac{3^{3} \cdot (4-A)^{(4-A)}}{A^{A} \cdot (1-A)^{2(1-A)} \cdot (2+A)^{(2+A)} \cdot (3-A)^{(3-A)}}\right]^{m}.
\end{align*}
We let
$q(A)=\ln 27+(4-A)\ln(4-A)-A\ln(A)-2(1-A)\ln(1-A)-(2+A)\ln(2+A)-(3-A)\ln(3-A),$
so that
\begin{align*} 
q'(A)=&-\frac{4-A}{4-A}-\ln(4-A)-\left[\frac{A}{A}+\ln(A)\right]-\left[-\frac{2(1-A)}{1-A}-2\ln(1-A)\right]\\
&-\left[\frac{2+A}{2+A}+\ln(2+A)\right]-\left[-\frac{3-A}{3-A}-\ln(3-A)\right]\\
=&-\ln(4-A)-\ln(A)+2 \cdot \ln(1-A)-\ln(2+A)+\ln(3-A).
\end{align*}
This expression is seen to equal zero (and yield a maximum) for $A=\frac{5}{2}-\frac{\sqrt{21}}{2}$.
Substituting this value into the expression ${m \choose j} \cdot {3m \choose m-j} \cdot {4m-j \choose m}$:
yields a maximum value that is $\cong 83.97^{m}$.
Stirling's approximation applied to the denominator yields
$$\frac{(4m)!}{(3m)!(m)!}\cong \left(\frac{256}{27}\right)^{m},$$
and thus, $$\pi\cong\frac{83.97^{m}}{\left(\frac{256}{27}\right)^{2m}}.$$
The Erd\H os-Lov\'asz local lemma with $d=O(n^2)$ and $\pi$ as above then yields
$$m=20.32\lg n(1+o(1)),$$ or $$k_0\le 81.28\lg(n)(1+o(1)),$$
as compared to the value $k_0\le 88.03$ given by the general bound (2).
\end{proof}

\subsection{$3$-covering Arrays with a $q$-letter Alphabet}
This section gives a generalization of Theorems 1 and 2 for an arbitrary alphabet size. 
\begin{thm}
$$k_0(n,q,3)\le B(q) \cdot \lg(n)(1+o(1)),$$
where the constant $B(q)$ is specified below.
\end{thm}
\begin{proof}
We first find a generalized expression for the probability $p$ of avoiding a particular word under a similar probability model as before:
$$p={\frac{\sum_{j=0}^{m} {m \choose j} \cdot {(q-1)m \choose m-j} \cdot {qm-j \choose m}}{{qm \choose m}^{2}}}.$$
The numerator summand can be written as
\begin{align*}
&{m \choose j} \cdot {(q-1)m \choose m-j} \cdot {qm-j \choose m}\\
=&\frac{m!}{j! (m-j)!} \cdot \frac{((q-1)m)!}{(m-j)! ((q-2)m+j)!} \cdot \frac{(qm-j)!}{m! ((q-1)m-j)!}\\
=&\frac{((q-1)m)! (qm-j)!}{j! (m-j)! (m-j)! ((q-2)m+j)! ((q-1)m-j)!}\\
\cong&\left(\frac{(q-1)m}{e}\right)^{(q-1)m} \cdot \left(\frac{(q-A)m}{e}\right)^{(q-A)m} \cdot \left(\frac{e}{Am}\right)^{Am}\cdot \left(\frac{e}{(1-A)m}\right)^{2(1-A)m} \\
&\cdot \left(\frac{e}{((q-2)+A)m}\right)^{((q-2)+A)m} \cdot \left(\frac{e}{((q-1)-A)m}\right)^{((q-1)-A)m}\\
=&\left[\frac{(q-1)^{(q-1)} \cdot (q-A)^{(q-A)}}{A^{A} \cdot (1-A)^{2(1-A)} \cdot ((q-2)+A)^{((q-2)+A)} \cdot ((q-1)-A)^{((q-1)-A)}}\right]^{m}
\end{align*}
Setting 
$r(A)=\ln (q-1)^{(q-1)}+(q-A)\ln(q-A)-A\ln(A)-2(1-A)\ln(1-A)-((q-2)+A)\ln((q-2)+A)-((q-1)-A)\ln((q-1)-A)$, we see that
\begin{align*}
r'(A)=&-\frac{q-A}{q-A}-\ln(q-A)-\left[\frac{A}{A}+\ln(A)\right]-\left[-\frac{2(1-A)}{1-A}-2\ln(1-A)\right]\\
&-\left[\frac{(q-2)+A}{(q-2)+A}+\ln((q-2)+A)\right]-\left[-\frac{(q-1)-A}{(q-1)-A}-\ln((q-1)-A)\right]\\
=&-\ln(q-A)-\ln(A)+2 \cdot \ln(1-A)-\ln((q-2)+A)+\ln((q-1)-A),
\end{align*}
and that $r'(A)=0$ if 
$$\ln\left(\frac{(1-A)^{2} \cdot (q-1-A)}{(q-A) \cdot A \cdot (q-2+A)} \right)=0,$$
or if
$$A^{2}-A(q+1)+1=0.$$
A feasible solution to this quadratic is 
\begin{equation}A=\frac{(q+1)- \sqrt{(q+1)^{2}-4}}{2}\end{equation}
Incorporating the denominator of the expression for $p$, we see that 
\begin{eqnarray}p&\cong& \left[\frac{(q-1)^{3(q-1)} \cdot (q-A)^{(q-A)}}{q^{2q}\cdot A^{A} \cdot (1-A)^{2(1-A)} \cdot (q-2+A)^{(q-2+A)} \cdot (q-1-A)^{(q-1-A)}}\right]^{m}\nonumber\\&:=&D^m,\end{eqnarray}
with $A$ given by (5) and with $\pi\le q^t\cdot p$. 
%Plugging $A$ value into the expression of p, where $C=\sqrt{(n+1)^{2}-4}$.\\
%Name $P_{1}=\left(\frac{(n-1+C)(1-n+C)(n-1)^{10}}{8(n+1-C)(3n-3-C)^{3}(n-3+C)n^{8}} \right)^{n}$, $P_{2}=\left(\frac{2(n-1+C)(1+n-C)(3n-3-C)}{(1-n+C)(n-3+C)} \right)^{C}$, $P_{3}=\left(\frac{8(3n-3-C)^{3}(n-3+C)^{3}}{(n-1)^{6}(n-1+C)(n+1-C)(1-n+C)}\right)$.
%\begin{align*}
%&\lg{\left(\frac{\left[\frac{(n-1)^{(n-1)} \cdot (n-A)^{(n-A)}}{A^{A} \cdot (1-A)^{2(1-A)} \cdot ((n-2)+A)^{((n-2)+A)} \cdot ((n-1)-A)^{((n-1)-A)}}\right]}{{nm \choose m}^{2}}\right)}\\
%=&\lg{\left(\frac{(n-1+C)(1-n+C)(n-1)^{10}}{8(n+1-C)(3n-3-C)^{3}(n-3+C)n^{8}} \right)^{\frac{n}{2}} }\\
%&+ \lg{\left(\frac{2(n-1+C)(1+n-C)(3n-3-C)}{(1-n+C)(n-3+C)} \right)^{\frac{C}{2}} }\\ 
%&+ \lg{\left(\frac{8(3n-3-C)^{3}(n-3+C)^{3}}{(n-1)^{6}(n-1+C)(n+1-C)(1-n+C)}\right)^{\frac{1}{2}} }\\
%=& \frac{1}{2} \cdot \left[\lg{P_{1}} + \lg{P_{2}} + \lg{P_{3}} \right]
%\end{align*}
Thus setting $e\pi d<1$, we obtain
$$k_0=qm\le  B(q)\lg n,$$
where $$B(q)=\frac{2q}{\lg(1/D)},$$ and with $D$ given by (6).
\end{proof}
\noindent REMARK:  A first order approximation to the maximizing value of $A$ is given by $A=\frac{1}{q+1}$; use of this approximation greatly streamlines the value of $p$ in (6), though computation of the optimal value of $p$ is not hard for any value of $q$.
\section{4-Covering Binary Arrays}
\begin{thm}
$$k_0(n,2,4)\le27.32 \cdot \lg(n) (1+o(1)).$$
\end{thm}
\begin{proof}
We first find the expression for the probability $p$ of avoiding a particular word (of the sixteen total) in a random equal weight array:  We set $k=4m$ and note that
$$p={\frac{\sum_{j=0}^{2m} {2m \choose j}{2m \choose j} \sum_{i=0}^{j} {j \choose i}{4m-j \choose 2m-i}{4m-i \choose 2m}}{{4m \choose 2m}^{3}}}.$$
The expression may be justified by multiplying and dividing by ${{4m}\choose{2m}}$ and arguing that the numerator represents the number of ways of avoiding the word 1111 in any four selected columns as follows:  We first select $2m$ ones in the first column in ${{4m}\choose{2m}}$ ways.  Then, for some $j$, we pick $j$ ones in the second column to correspond to the positions with a 1 in the first column.  We do the same for the positions with a 0 in the first column, choosing $2m-j$ of these.  For some $i$ we now pick $i$ ones in column 3 so as to form a 111.  Finally, we make sure that 1111 does not occur.  The rest of the proof follows the same steps as in the previous section.  Parametrizing by setting $j=Bn$, $i=ABn$, where $0 \leq A,B \leq 1$, we calculate that the summand $f(j,i)$ in the expression for $p$ equals

\begin{align*}
&f(j,i)={2m \choose j}{2m \choose j} {j \choose i}{4m-j \choose 2m-i}{4m-i \choose 2m}\\
=&\left(\frac{(2m)!}{j! (2m-j)!}\right)^{2} \cdot \frac{j!}{i! (j-i)!} \cdot \frac{(4m-j)!}{(2m-i)! (2m+i-j)!} \cdot \frac{(4m-i)!}{(2m)! (2m-i)!}\\
=&\left[\frac{2^{2} \cdot (4-B)^{(4-B)} \cdot (4-AB)^{4-AB}}{B^{B} \cdot (AB)^{AB} \cdot ((1-A)B)^{(1-A)B} \cdot (2-AB)^{2(2-AB)}}\right]^m\\ 
&\cdot \left[\frac{1}{(2-B)^{2(2-B)} \cdot (2+AB-B)^{2+AB-B}}\right]^{m}.
\end{align*}
We now find the value of $i$ for which the maximum occurs in the inner sum:
\begin{align*}
&{j \choose i}{4m-j \choose 2m-i}{4m-i \choose 2m} \\
=&\frac{j!}{i! (j-i)!} \cdot \frac{(4m-j!}{(2m-i)! (2m+i-j)!} \cdot \frac{(4m-i)!}{(2m)! (2m-i)!}\\
=&\left[\frac{B^{B} \cdot (4-B)^{(4-B)} \cdot (4-AB)^{4-AB}}{2^{2} \cdot (AB)^{AB} \cdot ((1-A)B)^{(1-A)B} \cdot (2-AB)^{2(2-AB)}}\right]^m\\ &\cdot \left[\frac{1}{(2+AB-B)^{2+AB-B}}\right]^{m}
\end{align*}
As before, we set $q(A)=(\ln (B^{B}/4))+(4-B)\ln(4-B)+(4-AB)\ln(4-AB)-AB\ln(AB)
-(B-AB)\ln(B-AB)-2(2-AB)\ln(2-AB)-(2+AB-B)\ln(2+AB-B)$,
so that
\begin{align*}
q'(A)=&-B-B\ln(4-AB)-(B+B\ln(AB))-(-B-B\ln(B-AB))\\
&-2(-B-B\ln(2-AB))-(B+B\ln(AB+2-B)).
\end{align*}
\\
Setting $q'(A)=0$ yields the critical value
$$A=\frac{3 - \sqrt{9-2B}}{B}.$$
Plugging the critical value of $A$ into the full expression for $f(j,i)$, we see that
\begin{eqnarray*}
f(j,i)&\le&\left[\frac{2^{2} \cdot (4-B)^{(4-B)} \cdot (1+\sqrt{9-2B})^{1+\sqrt{9-2B}}}{B^{B} \cdot (3-\sqrt{9-2B})^{3-\sqrt{9-2B}} \cdot (B-3+\sqrt{9-2B})^{B-3+\sqrt{9-2B}} }\right]^{m}\\
&\cdot& \left[\frac{1}{(\sqrt{9-2B}-1)^{2(\sqrt{9-2B}-1)}\cdot(2-B)^{2(2-B)} }\right]^{m}\\
&\cdot&\left[\frac{1}{(5-B-\sqrt{9-2B})^{5-B-\sqrt{9-2B}}}\right]^m.
\end{eqnarray*}
Repeating the same process, we set $r(B)=2\ln(2)+(4-B)\ln(4-B)+(1+\sqrt{9-2B})\ln(1+\sqrt{9-2B})-B\ln{B}
-(3-\sqrt{9-2B})\ln(3-\sqrt{9-2B})-(B-3+\sqrt{9-2B})\ln(B-3+\sqrt{9-2B})
-(2\sqrt{9-2B}-2)\ln(\sqrt{9-2B}-1)-(4-2B)\ln(2-B)
-(5-B-\sqrt{9-2B})\ln(5-B-\sqrt{9-2B})$, and set $r'(B)=0$ to obtain the critical value (using Maple) of
$B \approx 0.912621974615847$.  Since
$A=\frac{3 - \sqrt{9-2B}}{B}$, we get
$A \approx 0.352201128737$,
and plugging these values of $A$ and $B$, we get the numerator of $p$ bounded by
$m^2\cdot(e^{8.013})^m$.  Since the denominator expression is $\cong 16^{3m},$ we get that
\[\pi\le16p\cong \lr\frac{e^{8.013}}{16^3}\rr^m.\]
Since $d=O(n^3)$, the Lov\'asz lemma yields
that a suitable 4-covering array exists if 
$$m>\frac{3\lg(n)}{\lg(1.3558)} \approx 6.83082\lg(n),$$
and thus
$$k_0\le 4(6.83082)\lg n=27.32\lg(n).$$
\end{proof}

\noindent REMARK:  Our upper bound of $27.32\lg n$ should be compared to the bound of $32.22\lg n$ as given by (2).  Also, the analysis in this section can readily be extended to $q$-ary 4-covering arrays, $q\ge3$, but we do not provide details.

\section{Acknowledgement}  The research of AG and ZK was supported by NSF grant 1263009.  RY also participated in the project without NSF support but with a great level of enthusiasm.

\end{document}